\documentclass[12pt,epsfig]{amsart}
\usepackage{latexsym}
\usepackage{amsmath,amscd}
\usepackage{graphics}

\usepackage{epsfig}

\usepackage{amssymb}

\usepackage{latexsym}

\include{plaquettemacro}

\usepackage{hyperref}

\newtheorem{theorem}{Theorem}[section]
\newtheorem{prop}{Proposition}[section]
\newtheorem{lemma}[theorem]{Lemma}

\newtheorem{remark}[theorem]{Remark}

\numberwithin{equation}{section}

\newcommand{{\tlg}}{\tilde\gamma }

\newcommand{{\tlG}}{\tilde\Gamma }
\newcommand {{\tlJ}}{\tilde J }

\def\slash#1{\, /\kern-0.6em{#1}}

\begin{document}

\title{Geometric Analysis on Path spaces}

\author{Pradip Kumar}
\address{{\rm Pradip Kumar}\\ Department of Mathematics,
  Harish Chandra Research Institute\\ Allahabad 211019, Uttarpradesh\\ India} \email{pmishra@hri.res.in}

\title{Geometric Analysis on Path spaces}




\begin{abstract}
Let $M$ be any $n$ dimensional smooth manifold  and $PM$ be the space of all smooth paths, then we showed that $PM$ is a smooth manifold modelled over a complete
normable space.
We discussed many geometric structure on Path spaces and its relation to ambient space.
\end{abstract}
\maketitle
\section{Introduction}

Let M be any smooth manifold of dimension $n$. We define path space $PM$ over $M$ as,$PM=\{\gamma:[0,1]\to M : \gamma \text{ smooth }\}$.

In Section 2, we showed that $PM$ is a smooth manifold modelled over a complete normable space (Note that: We are not saying Banach space, as we can not
specify norm). In \cite{Michor}, Chapter 10, Theorem 10.4, Michor showed that path spaces with $\text{ fine }-\mathfrak{D}$ ($F\mathfrak{D}$) topology is $C_c^\infty$ manifold modelled on nuclear and
dually nuclear locally convex vector space. Our manifold structure on $PM$ is easy to work as it is modelled over complete normable space. ODE has unique existence and as
result unique geodesic and other routine thing can be seen in path space.

In section 3, we will partially follow \cite{Eliasson}. We defined connection and curvature on $PM$ and its relation with ambient manifold's connection and curvature.
In theorem ~\ref{theo:completeness} we proved if M is complete with respect to a connection $\nabla$ then $PM$ is complete with respect to induced connection
which we defined.
\section{Path Space as a Complete Normable Space}\label{s:normable}
Let $M$ be a smooth manifold, and $PM$ is path space, For $\gamma \in PM$ define
$X_{\gamma }:=\{X:\gamma([0,1])\to TM,\text{  }  \text{smooth vector field along } \gamma \}$. Then $X_\gamma $ is a vector space.
Next,  we will define a structure of complete normable space to $X_\gamma $.

Case-I: If $\gamma $ contained in chart domain $(U,\phi)$ of $M$.
 Identifying $X_{\gamma} $ as the set $\{\mu :[0,1]\to \mathbb R^n , \text{ smooth } \}$,
we define
\begin{equation}\|\mu\|= \text{ Sup }_{t\in[0,1]} \text{Max}(|\mu(t)|,|\mu'(t)|)\label{norm}
\end{equation}
where $|.|$ is usual euclidean norm. We claim that equivalence class of this norm does not depend upon the choice of chart $(U,\phi)$. For if there are two coordinate system $(U,\phi)$ and $(V,\psi)$, such
that $\gamma([0,1])\subset U\cap V$. Without loss of generality, We can assume that $V=U$. Let $u:[0,1] \to \mathbb R^n$ is representation of $X\in X_\gamma$ via coordinate map $\phi$. And
Let $ v :[0,1] \to \mathbb R^n$ is representation of $X\in X_\gamma$ via coordinate map $\psi $ . If $\lambda (=\phi o\psi^{-1})$ is transition map, then for each $t\in [0,1]$
\begin{equation}v(t)= d\lambda_{\phi(\gamma(t))} u(t)\text{ and this gives } \label{v and u relation}\end{equation}
\begin{equation}|v(t)|\leq |u(t)|. |d\lambda_{\phi(\gamma(t))}|\label{first comprasion}
\end{equation}
 Now $v'(t)= d\lambda _{\phi(\gamma(t))} u'(t) + d^2\lambda_{\phi(\gamma(t))}(d\phi_{\gamma(t)}(\gamma'(t),u(t)))$. This gives,
\begin{equation}|v'(t)|\leq |u'(t)|.|d\lambda_{\phi(\gamma(t))}|+|u(t)|.|d\phi_{\gamma(t)}(\gamma'(t))|.|d^2\lambda_{\phi(\gamma(t))}|\label{second comprasion}
\end{equation}
From equations \ref{first comprasion} and \ref{second comprasion}
\begin{equation}\text{ max } (|v(t)|,|v'(t)|)\leq \text{ max }(|u(t)|. |d\lambda_{\phi(\gamma(t))}|,|u'(t)|.|d\lambda_{\phi(\gamma(t))}|+|u(t)|.|d\phi_{\gamma(t)}(\gamma'(t))|.|d^2\lambda_{\phi(\gamma(t))})|\label{third comprasion}\end{equation}
Observe $\text{ Max} |d\lambda\text{ on }\phi(\gamma[0,1])|$ and $\text{ Max}|d^2\lambda\text{ on }\phi(\gamma[0,1])| $ are constant which depend upon the charts and independent of $u$ and $\gamma $.
From equations \ref{norm}, \ref{third comprasion} we have $\|v\|\leq c\|u\|$, for some constant $c$, which will depend upon only coordinate map. Same way we have $\|u\|\leq c'\|v\|$, for some constant $c'$.  Hence for any $X\in X_\gamma$ such that $\gamma $ lies in one coordinate chart, we defined norm for every representation and each norm is equivalent.

Case -II: General case: When $\gamma$ does not lie in single coordinate chart then, we cover Im$\gamma$ by finitely many charts $U_1,\cdots,U_n$ and then partition $[0,1]$ in compact subset $K_1,\cdots,K_n$ with $\gamma((K_i)\subset U_i$. We take the norm on $K_i$ associated to $U_i$ as in Case-I and then take their sum.  These norms will be equivalent.

In this way we defined a structure of complete normable space on $X_\gamma$.  Now we will move to get an
injective map from some open set of $X_\gamma$ to $PM$,  then we will prove this map becomes coordinate map.
We can identify $X_\gamma $ as $\Gamma (\gamma^*(TM))$, where $\gamma^*(TM)$ is a pull back
vector bundle over $[0,1]$ by $\gamma$ and $\Gamma(\gamma^*(TM)$ is set of all smooth
section $f:[0,1] \to \gamma^*(TM)$ [Note: We are talking about pull back vector bundle over a corner manifold,
for this we are following \cite{Michor}]. For each $\gamma\in PM$,
we have one $\widetilde{M}$ Riemannian manifold, submanifold of $M$,
such that $\gamma([0,1])\in \widetilde{M}$ and we have $\gamma^*(T\widetilde{M})=\gamma^*(TM)$.
 Also We have following well known theorm of Riemannian geometry:
\begin{theorem}\label{thm:riemannian geometry thm}
 Let $M$ be a Riemannian manifold. Then there exits $\epsilon >0$, denote by $O_{\epsilon}$, the open $\epsilon $ neighborhood
of $0-$ section of $\tau: TM \to M$, that is $O_\epsilon =\{\zeta\in TM:\|\zeta\|<\epsilon\}$ such that
\begin{equation}
(\tau, exp): O_{\epsilon} \to M\times M\label{riemannian geometry thm}
\end{equation}
is a diffeomorphism  onto an open neighborhood of the diagonal of $M\times M$.
\end{theorem}
Applying above theorem for $\widetilde{M}$, We can get $O_\gamma$ some open set in
$\gamma^* (T\widetilde{M})$, containing zero section such that
$\pi^*\gamma(O_\gamma)\subset O_\epsilon$, Where
$\pi^*\gamma :\gamma^* (T\widetilde{M})\to (T\widetilde{M})$ is inclusion. We can identify $\pi^*\gamma(O_\gamma)$ as $O_\gamma$. Also from ~\ref{riemannian geometry thm} we have
\begin{equation}
exp_\gamma:=exp:O_\gamma \to exp(O_\gamma)\subset M\label{exp map}\end{equation} is diffeomorphism.
 If $O_\gamma$ is
open in $TM$, then $\Gamma(O_\gamma)$ is open in $X_\gamma$ with topology of complete normable structure. Keeping in mind above identification,for $\gamma \in PM$, define
\begin{eqnarray}
&&Exp_\gamma:\Gamma(O_\gamma)\to PM\text{ by}\nonumber\\
&&X\mapsto exp_{\gamma}(X);\text{ }X(t)= exp_{\gamma(t)}X(t)\label{coordiante chart}
 \end{eqnarray}
If $U_\gamma = Exp_\gamma (\Gamma(O_\gamma))$, we claim
$\{(U_\gamma, Exp_\gamma)\}$ is coordinate chart. Injectivity is obivious here.
Let $\gamma_1$ and $\gamma_2$ be two path in $PM$, such that if $Exp_{\gamma_1} (\Gamma(O_{\gamma_1}))= U_{\gamma_1}$ and
$Exp_{\gamma_2} (\Gamma(O_{\gamma_2}))= U_{\gamma_2}$ and $U_{\gamma_1}\cap U_{\gamma_2}\neq \phi$
 Then we have a map:
 $$Exp_{\gamma_1}^{-1}(U_{\gamma_1}\cap U_{\gamma_2})\subset \Gamma(O_{\gamma_1})\to Exp_{\gamma_2}^{-1}(U_{\gamma_1}\cap U_{\gamma_2})\subset \Gamma(O_{\gamma_2})$$
 $$X\mapsto Exp_{\gamma_2}^{-1}oExp_{\gamma_1}(X)$$
 We will show above map is smooth map from open set of complete normable space $X_{\gamma_1}$ to open set of $X_{\gamma_2}$.
 Without loss of generality we can assume
 $Exp_{\gamma_1}^{-1}(U_{\gamma_1}\cap U_{\gamma_2})= \Gamma(O_{\gamma_1})$  and  $Exp_{\gamma_2}^{-1}(U_{\gamma_1}\cap U_{\gamma_2})= \Gamma(O_{\gamma_2})$.

\begin{lemma}\label{lem:smoothness}
 Let $V$ and $W$ be vector bundle over $[0,1]$ (vector bundle over corner manifold as in \cite{Michor})
and $p: U\subset V\to W$ be a smooth map which takes points in the fiber of $V$ over a point $t\in[0,1]$ into
points in the fiber of $W$ over the same point $t$.  Define a non linear operator $P:\widetilde{U}\subset \Gamma(V)\to \Gamma(W)$ by
 $P(f)= p o f$
with the complete normable structure on $\Gamma(V)$ and $\Gamma(W)$ as before,  $P$ is  smooth map.
\end{lemma}
\begin{proof} If $y$ is coordinate chart in the finber of $V$, then $P$ can be written as $[P(f)](t)= p(t,f(t))$ and follwoing \cite{Hamilton}
$[DP(f)h](t)= D_y p(t,f(t)) h(t)$
 and using same line as in I.3.1.7 of \cite{Hamilton} ,It follows that $P$ is $C^\infty$,$G\hat{a}teaux$ [Note: In refrence\cite{Hamilton} in $\Gamma(V)$, Frechet space structure is not same as ours, still proof
is same]. Also as each $G\hat{a}teaux$ derivative is linear and bounded at each point of $\widetilde{U}$, we have P is frechet smooth in complete normable space.\end{proof}
Now with $p= exp_{\gamma_2}^{-1} o exp_{\gamma_1}$, $V= X_{\gamma_1}$ and $W= X_{\gamma_2}$ as vector bundle over $[0,1]$. From \ref{exp map}, $p$ is smooth
from $O_{\gamma_1}\to O_{\gamma_2}$. Also $P(X)= Exp_{\gamma_2}^{-1}oExp_{\gamma_1}(X)= poX$. By lemma \ref{lem:smoothness}, $P$ is smooth map,
which in turn shows that $PM$ is a smooth manifold modelled on complete normable space.

As in Banach manifold, here also we have Tangent space and Tangent bundle.
Tangent space at point $\gamma\in PM$ will be $X_\gamma$ and Tangent bundle is $\coprod_\gamma X_\gamma= T(PM)$
\section{Vector Field, Connection and Curvature}\label{s:vectcon}
We define a path on path space by a continuous map
\begin{equation}
\Gamma: [0,1]\rightarrow PM;  \text{  }s\mapsto \Gamma(s)\in PM.\label{path}
\end{equation}
Thus for each $s\in [0,1]$ we have a path given by
$$\Gamma(s)(t):=\Gamma_s(t):=\Gamma(s,t),$$
A vector field on path space is a smooth map
\begin{equation}
V: PM\to T(PM)\label{vector field}
\end{equation}
Tangent vector field along this path $\Gamma$ is given by
\begin{equation}
\Gamma':[0,1]\to T(PM);\text{ }s\mapsto \frac{\partial}{\partial s} \Gamma(s,t)\label{tgt vectorfield}
\end{equation}
Hence if $V$ is a vector field along $\Gamma$ then $V(s)$ is vector field along $\Gamma(s)$
and also if we define for each $t$
\begin{equation}
ev_t:PM\to M,\label{evaluation}
\end{equation}
 which is given by $ev_t(\gamma) =\gamma(t)$. This is a
smooth map and $d(ev_{t_0})(\gamma)(V)= V(\gamma(t_0))$. So we can identify vector field along $\Gamma$ as vector field along each path with $V(s,t)\in T_{\Gamma(s,t)}M$.

If K is a vector field on $M$, then we have $\widetilde{K}$ vector field on $PM$ defined by
\begin{equation}
\widetilde{K}(\gamma)(t)= K(\gamma(t))
\end{equation}
but conversely we can not define. Now if $K_1$ and $K_2$ are vector field on $PM$  and $M$ has a connection
\begin{equation}
\nabla:\chi(M)\times \chi(M)\to \chi(M),
\end{equation}
 and we want to define induced covariant connection $\widetilde{\nabla}$ on $PM$.

Let $X$ and $B$ be manifold, $\pi :X\to B $ is a vector bundle over base B,
then a connection on this bundle is a smooth map \begin{equation}\phi: T(B)\times_B X\to T(X)\end{equation}
such that for each $b\in B$ and $\xi \in T_b(B)$ and $x\in X$ with $\pi(x)=b$, then $\phi(\xi,x)\in T_x(X)$ and $d\pi_x (\phi (\xi,x))= \xi$.
We will denote this as $\phi(b)(\xi, x)$ for emphasing the value of base poin $b$.
Now let M has a connection $\phi: (TM)\times_M (TM)\to TM$, then for path space we have the following induced connection,
for each $t\in [0,1]$, define
$$\widetilde{\phi} : T(PM)\times_{PM} T(PM) \to T(PM)$$
such that for each $t\in [0,1]$, we have
\begin{equation}
\widetilde{\phi}(V,W)(\gamma)(t):= \phi(\gamma(t))(V(t),W(t))\label{connection on path space 1}
\end{equation}
with $V(t)\in T_{\gamma(t)} M$ and $W(t)\in T_{\gamma(t)}M$.\\
Now for $V$ and $W\in T_\gamma(PM)\sim X_\gamma$, covariant connection on $PM$ is given by
\begin{eqnarray} &&\widetilde{\nabla}(V,W)=d(W)_\gamma(V)-\widetilde{\phi}(V,W)(\gamma)\nonumber\\
 &&\widetilde{\nabla}(V,W)(\gamma)(t)= dW_{\gamma(t)}(V(t))-\phi(\gamma(t))(V(t),W(t))\label{conn}
 \end{eqnarray}
With this point-wise definition of covariant connection, it is obvious that
$\Gamma$ is a geodesic in $PM$ if and only if transrvrese path $\Gamma_t$
as defined in \eqref{trans} is a geodesic in $M$
for each $t \in[01,] $. Thus we make the following proposition
\begin{prop}\label{prop:initialcongeo}
For any given $\gamma\in PM$ and any vector $V\in T_{\gamma}PM$, there is a
unique geodesic $\Gamma:[0,b]\rightarrow PM$, such that
$\Gamma(0)=\gamma$ and $\Gamma'(0)=V$.
\end{prop}
\begin{proof}Let $V\in T_{\gamma}PM$ is given by $V(t)\in T_{\gamma(t)}M,$ for each $t\in [0,1]$.
Now for each $t\in[0,1]$ we have the following initial conditions
$\Gamma_{t}(0)=\gamma(t)$ and $\Gamma'_{t}(0)=V(t)$. With thsese initial conditions, for each $t\in[0,1]$ we have a unique geodesic $\Gamma_t$ defined on small 
interval say $[0,\epsilon(t)]$ where $\epsilon(t)>0$. As $[0,1]$ is compact,we have
a unique geodesic $\Gamma$ on $PM$ on small interval with initial conditions $\Gamma(0)=\gamma$ and $\Gamma'(0)=V$.
\end{proof}
From proposition~\ref{prop:initialcongeo} it follows that:
\begin{theorem}\label{theo:completeness}
If $M$ is complete with respect to a connection $\nabla$ then $PM$ is
complete with respect to induced connection $\widetilde{\nabla}$
\end{theorem}
Similarly, we can define the curvature on $PM$. If $R$ is a curvature of a manifold $M$ then,
$\widetilde{R}$ curvature of $PM$ is defined as following:  For given $X$ and $Y\in \chi(PM)$, We have

\begin{eqnarray}
&&\widetilde{R}(X,Y):\chi(PM)\to \chi(PM)\nonumber\\
&&\widetilde{R}(X,Y)(Z)(\gamma)(t):= R(X(\gamma(t)), Y(\gamma(t)))(Z(\gamma(t)))\label{curvature}
\end{eqnarray}
Following, we have some weaker version of Hopf-Rinow theorem.
\begin{theorem}\label{thm: HopfRinow-1}
Let $M$ be any simply connected riemannian manifold with non positive sectional curvature , Then for given point $\gamma_1$ and $\gamma_2$ in connected component in $PM$ there always exits a path
space geodesic joining these two points.
\end{theorem}
\begin{proof}
 If $\gamma_1$ and $\gamma_2\in PM$ then for $t\in [0,1],\gamma_1(t)$ and $\gamma_2(t)\in M$. By Cartan-Hadamard theorem 3.1,\cite{Docarmo} Page 149,  For each $t\in[0,1]$,
$exp_{\gamma_1(t)}: T_{\gamma_1(t)}\to M$ diffeormorphism, so for $\gamma_2(t)\in M$ there exists $V_t\in T_{\gamma_1(t)}$, such that $exp_{\gamma_1(t)}V_t= \gamma_2(t)$.
Now define, \begin{equation} \Gamma:[0,1]\to PM; \text{ by } \Gamma(s)(t):= exp_{\gamma_1(t)}s.V_t\label{geodesic eqn}
            \end{equation}
This is a geodesic, because for each $t\in[0,1]$, $\Gamma_t(s)$ is a geodesic in $M$.
\end{proof}

\begin{remark}
Above theorem is not true for positive sectional curvature manifold. For example, take $M= S^2$, then
there exists $\gamma_1$ and $\gamma_2$ in connected component of  $PS^2$(path space over $S^2$), which can not be joined by path space geodesic.
For this purpose, Let N and S are respectively north pole and south pole of $S^2$. See figure below.
Let $NAXBS$  and $NCYDS$ be part of great circle, Let $\gamma_1$ be smooth path  in $S^2$whose image is path $ANC$ on great circles such that speed at $N$ is zero.
\includegraphics[scale=0.55]{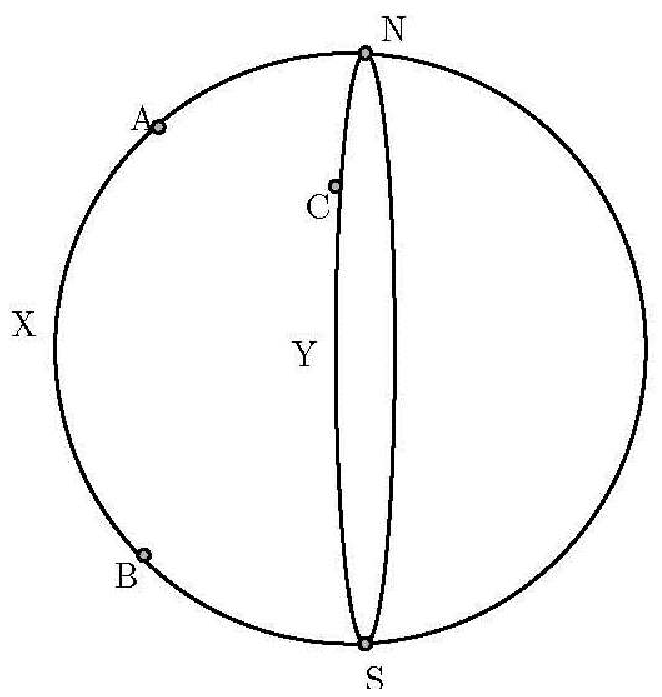}.Hence $\gamma_1$ has corner in $N$, but it is smooth (Obviously, it is not analytic). Now take $\gamma_2$ be smooth path in $S^2$ whose image is path $BSD $ on
great circles such that speed at $S$ is zero. Hence $\gamma_2$ has corner in $S$. These is a homotopy which between these two paths, hence these two points in $PS^2$ are
in connected component, but there does not exists any homotopy which is also a path space geodesic. For if path space geodesic $\Gamma(s)(t)$ exists joining
$\gamma_1$ and $\gamma_2$, as image of $\Gamma(s)(t)$ always lies on great circles $NAXBS$  and $NCYDS$, and then continuity of $\Gamma$ breaks and $N$ and $S$
\end{remark}

\medskip

{\bf Acknowledgments:}
I would like to thank Prof. Joesh Oesterl\`{e} profusely for the very useful discussions. I would also thank Dr. Rukmini Dey for useful discussions.


\end{document}